\documentclass[letterpaper, 10 pt, conference]{ieeeconf}

\IEEEoverridecommandlockouts                              
 \pdfoutput=1
\usepackage{amsmath, amssymb}
\usepackage{amsfonts}
\usepackage{graphicx}
\usepackage{enumerate}
\usepackage[font = small]{caption}
\usepackage{ifthen}
\usepackage{multicol}
\usepackage{float}
\usepackage{cite}
\usepackage{fancyhdr}
\usepackage[usenames, dvipsnames]{color}
\usepackage[lofdepth,lotdepth]{subfig}
\graphicspath{{figures/}}
\usepackage{mathtools}

 \usepackage{tikz}
 \usetikzlibrary{decorations.pathmorphing,shadings}
\usepackage{tkz-graph}
\usetikzlibrary{arrows}

 \usepackage{cancel} 

\newtheorem{theorem}{Theorem}
\newtheorem{lemma}[theorem]{Lemma}

\newtheorem{rem}{Remark}

\newtheorem{ass}{Assumption}
\newtheorem{definition}{Definition}

\newboolean{showcomments}
\setboolean{showcomments}{true}

\newcommand{\todo}[1]{  \ifthenelse{\boolean{showcomments}}
{\textcolor{ForestGreen}{TO DO:  #1}}{}}
\newcommand{\bassam}[1]{\ifthenelse{\boolean{showcomments}}
{\textcolor{Orange}{(Bassam says: #1)}}{}}
\newcommand{\henrik}[1]{\ifthenelse{\boolean{showcomments}}
{\textcolor{Blue}{(Henrik says: #1)}}{}}
\newcommand{\partha}[1]{\ifthenelse{\boolean{showcomments}}
{\textcolor{Blue}{(Partha says: #1)}}{}}
\newcommand{\emma}[1]{\ifthenelse{\boolean{showcomments}}
{\textcolor{VioletRed}{(Emma says: #1)}}{}}
\newcommand{\ifneeded}[1]{\ifthenelse{\boolean{showcomments}}
{\textcolor{Gray}{#1}}{}}

\newboolean{showedit}
\setboolean{showedit}{true}
\newcommand{\edit}[1]{\ifthenelse{\boolean{showedit}}
{\textcolor{Blue}{#1}}{}}
\newcommand{\draft}[1]{\ifthenelse{\boolean{showedit}}
{\textcolor{gray}{#1}}{}}

\newcommand{\ddt}{\frac{\mathrm{d}}{\mathrm{d}t}}

\newcommand{\nth}{$n^{\text{th}}$ }
\newcommand{\rel}{\mathrm{Re}\{\lambda_2\}}

\newcommand{\rell}{\mathrm{Re}\{\lambda_l\}}
\newcommand{\imll}{\mathrm{Im}\{\lambda_l\}}

\usepackage{array}
\newcolumntype{L}[1]{>{\raggedright\let\newline\\\arraybackslash\hspace{0pt}}m{#1}}
\newcolumntype{C}[1]{>{\centering\let\newline\\\arraybackslash\hspace{0pt}}m{#1}}
\newcolumntype{R}[1]{>{\raggedleft\let\newline\\\arraybackslash\hspace{0pt}}m{#1}}
\usepackage{subfig}

\usepackage{tikz}
 \usetikzlibrary{plotmarks}
 \usepackage{pgfplots}

\usetikzlibrary{circuits.ee.IEC, positioning, external}

\definecolor{gray3}{rgb}{0.75, 0.75, 0.75}
\definecolor{gray2}{rgb}{0.5, 0.5, 0.5}
\definecolor{gray1}{rgb}{0.25, 0.25, 0.25}
\definecolor{gray0}{rgb}{0.15, 0.15, 0.15}

\definecolor{emmagreen1}{rgb}{0, 0.5, 0.1}
\definecolor{emmaorange1}{rgb}{0.99, 0.5, 0}
\definecolor{emmablue1}{rgb}{0, 0.25, 0.5}
\definecolor{emmaskyblue1}{rgb}{0.4, 0.8, 0.95}
\definecolor{emmapurple1}{rgb}{0.5, 0.25, 0.6}

\pgfplotscreateplotcyclelist{mycolorlist3}{%
emmagreen1\\%
emmaorange1\\%
emmablue1\\%
emmapurple1\\%
emmaskyblue1\\%
}

\pgfplotscreateplotcyclelist{mygraylist}{%
black\\%
gray1, dashed\\%
gray3\\%
gray2\\%
gray2,dashed\\%
gray3,dashed\\%
}

\renewcommand{\baselinestretch}{0.98}

\title{\LARGE \bf Localized high-order consensus destabilizes large-scale networks}

\author{ {Emma Tegling, Bassam Bamieh and Henrik Sandberg} 
 \thanks{ E. Tegling and H. Sandberg are with the School of Electrical Engineering and Computer Science, KTH Royal Institute of Technology, SE-100 44 Stockholm, Sweden {(\tt tegling, hsan@kth.se)}.} \thanks{ B. Bamieh is with the Department of Mechanical Engineering at the University of California at Santa Barbara, Santa Barbara, CA, USA, 93106. {\tt (bamieh@engineering.ucsb.edu)}.}
  \thanks{This work was supported in part by the Swedish Research Council through grants 2013-5523 and 2016-00861.}
}

\begin{document}
\maketitle


\begin{abstract}
We study the problem of distributed consensus in networks where the local agents have high-order ($n \ge 3$) integrator dynamics, and where all feedback is localized in that each agent has a bounded number of neighbors. 
 We prove that no consensus algorithm based on relative differences between states of neighboring agents can then achieve consensus in networks of any size. 
That is, while a given algorithm may allow a small network to converge to consensus, the same algorithm will lead to instability if agents are added to the network so that it grows beyond a certain finite size. This holds in classes of network graphs whose algebraic connectivity, that is, the smallest non-zero Laplacian eigenvalue, is decreasing towards zero in network size. This applies, for example, to all planar graphs.  Our proof, which relies on Routh-Hurwitz criteria for complex-valued polynomials, holds true for directed graphs with normal graph Laplacians. We survey classes of graphs where this issue arises, and also discuss leader-follower consensus, where instability will arise in any growing, undirected network as long as the feedback is localized. 
%
\end{abstract}

\section{Introduction}
The problem of distributed coordination of networked systems is one of the most active research topics in the field. In particular, since the seminal works by Fax and Murray~\cite{FaxMurray}, Olfati-Saber and Murray~\cite{OlfatiSaber2004}, and Jadbabaie \emph{et al.}~\cite{Jadbabaie2003} much effort has been directed to the sub-problem of distributed consensus. The consensus objective is, simply put, for the agents in the network to reach a common state of agreement. The applications range from distributed computing and sensing to power grid synchronization and coordination of unmanned vehicles~\cite{OlfatiSaber2007}. 

In most cases, the literature has focused on first-order algorithms, or information consensus, or second-order algorithms, which apply to moving point masses. However, higher-order algorithms, which are the focus of the present work, have also received attention, for example in~\cite{Ren2006, Ren2007, Ni2010, Rezaee2015,Jiang2009,Zuo2017}. 
Here, the local dynamics of each agent is modeled as an \nth order integrator ($n \ge 3$), and the control signal -- the consensus algorithm -- is a weighted sum of relative differences between states of neighboring agents. This can be viewed as an important theoretical generalization of the first- and second-order algorithms~\cite{Jiang2009}, but also has practical relevance. For example, not only position and velocity, but also acceleration feedback play a role in flocking behaviors, leading to a model where $n = 3$~\cite{Ren2006}. 

Typically, the research problem in focus is that of convergence of a given set of agents to consensus, and its dependence on various properties of the network. For example, directed communication, a switching or random topology~\cite{Ni2010}, or a leader-follower structure~\cite{Zuo2017}. In this paper, we take a different perspective and inquire as to the \emph{scalability} of a given consensus algorithm to large networks. That is, assuming that the interaction rules between agents are fixed, can the network be allowed to grow by adding new agents? This scenario is treated in~\cite{Bamieh2012} for localized first- and second-order consensus problems, proving asymptotic (in network size) network dimension-dependent bounds on global performance. Similar problems were addressed in~\cite{Patterson2014, SiamiMotee2015, Tegling2017b}. Those works focused on the \emph{performance} of the consensus algorithm. We show in this paper that the problem of high-order consensus is more fundamental: can \emph{stability} be maintained as the network~grows?


The result we present in this paper is clear-cut: the high-order ($n\ge 3$) consensus algorithm treated in, e.g.,~\cite{Ren2006, Ren2007} can \emph{not} allow the network to scale in graphs where the algebraic connectivity is decreasing towards zero in network size. We prove that at some finite network size, the closed-loop stability criteria will inevitably be violated, rendering the consensus algorithm \emph{inadmissible} in our terminology. 

The algebraic connectivity, that is, the smallest non-zero eigenvalue of the weighted network graph Laplacian, decreases towards zero in classes of graphs where the interactions are \emph{localized}, in that the size of each agent's neighborhood is bounded. We show this property for lattices and their fuzzes and subgraphs, planar and constant-genus graphs as well as growing tree graphs, building on existing results on their algebraic connectivities. In leader-follower consensus over undirected graphs, the locality property alone is sufficient to cause inadmissibility. This latter result was shown by Yadlapalli \emph{et al.} in~\cite{Swaroop2006} using a different method than in our work. Here, we generalize their result to leaderless consensus and directed, weighted graphs. 

To the best of our knowledge, the inability of the high-order consensus algorithm to achieve consensus in networks of any size has not been observed in literature apart from the result in~\cite{Swaroop2006}.
While it is noted in~\cite{Ren2007, Jiang2009} that the controller gains must be chosen with care to ensure stability, we point out that such a choice can only be done with knowledge of the algebraic connectivity -- a global network property. Consensus is, however, a distributed controller, and as such should preferably be possible to design and implement in a distributed fashion, without knowledge of global properties. Our result shows that this is not possible for the high-order consensus algorithm. 

The remainder of this short paper is organized as follows. We next introduce the \nth order consensus algorithm and the network model. In Section~\ref{sec:mainresult} we introduce our main result and discuss classes of graphs where it applies. We give numerical examples in Section~\ref{sec:numerical} and conclude in Section~\ref{sec:discussion}.

\section{Problem setup}
\label{sec:setup}
We begin by introducing the modeling framework for the \nth order consensus algorithm. The algorithm we consider adheres to the ones considered in~\cite{Ren2006,Ren2007,Ni2010, Rezaee2015} and is a straightforward extension to the better-known standard first- and second-order consensus algorithms.

\subsection{Network model and definitions}
Consider a graph $\mathcal{G} = \{\mathcal{V}, \mathcal{E}\}$ with $N = |\mathcal{V}|$ nodes. The set $\mathcal{E} \subset \mathcal{V} \times \mathcal{V}$ contains the edges, each of which has an associated nonnegative weight $w_{ij}$. 
We will in general let the graph be directed, so that the edge $(i,j)$ points from node~$i$ (the tail) to node~$j$ (the head).  The neighbor set $\mathcal{N}_i$ of node~$i$ is the set of nodes~$j$ to which there is an edge $(i,j)\in \mathcal{E}$. 
The outdegree of node~$i$ is defined as $d^+_i = \sum_{j =1}^N w_{ij}$ and its indegree is $d^-_i = \sum_{j =1}^N w_{ji}$ ($w_{ij}=0$ if $(i,j)\notin \mathcal{E}$). 
The graph~$\mathcal{G}$ is \emph{balanced} if $d^+_i = d^-_i$ for all $i \in \mathcal{V}$ and  \emph{undirected} if $(i,j) \in \mathcal{E} \rightarrow (j,i) \in \mathcal{E}$ for all $i,j \in \mathcal{V}$ and $w_{ij} = w_{ji}$. It is \emph{strongly connected} if there is a directed path connecting any two nodes $i,j \in \mathcal{V}$ and 
has a \emph{connected spanning tree} if there is a path from some node $i\in \mathcal{V}$ to any other node~$j \in \mathcal{V}\backslash\{i\}$.

The weighted graph Laplacian~$L$ is defined as follows:
\vspace{-0.5mm}
\begin{equation}
\label{eq:laplaciandef}
[L]_{ij} = \begin{cases}  -w_{ij}& \mathrm{if}~ j \neq i~\mathrm{and}~j \in \mathcal{N}_i
\\ 
\sum_{k \in \mathcal{N}_i} w_{ik} & \mathrm{if}~ j = i\\
0 & \mathrm{otherwise.}
 \end{cases}
\end{equation}
\vspace{-0.5mm}
By this definition, $L = D-A$, where $D = \mathrm{diag}\{d_1^+,d_2^+, \ldots, d_N^+\}$ is the diagonal matrix of outdegrees and $A$ is the adjacency matrix of the graph. Denote by $\lambda_l$, $l = 1,\ldots,N$ the eigenvalues of~$L$. Zero will be a simple eigenvalue of~$L$ if the graph has a connected spanning tree, which is what we assume henceforth. Remaining eigenvalues are in the right half plane (RHP). They are numbered so that $0 = \lambda_1 < \mathrm{Re}\{\lambda_2\}\le \ldots \le \mathrm{Re}\{\lambda_N \}$. 

In this paper, we will assume that the graph Laplacian~$L$ is \emph{normal} which requires that $L^TL = LL^T$. This means that $L$ is unitarily similar to a diagonal matrix. If the graph is undirected, $L$ is symmetric and thus always normal. For a directed graph, the normality of $L$ implies that $\mathcal{G}$ is balanced.

\subsection{\nth order consensus}
The local dynamics at each node~$i\in \mathcal{V}$ is modeled as a chain of $n$ integrators:
\begin{align*}
\ddt x_i^{(0)}(t)  &= x_i^{(1)}(t)\\
& ~\vdots \\
\ddt x_i^{(n-2)}(t)  &= x_i^{(n-1)}(t) \\ 
\ddt x_i^{(n-1)}(t)  &= u_i(t)
\end{align*}
%
where we let the information state $x_i(t)  \in \mathbb{R}$ (see Remark~\ref{rem:dimension}). 
The notation for time derivatives is such that $x_i^{(0)}(t) = x_i(t)$, $x_i^{(1)}(t) = \ddt x_i(t)  = \dot{x}_i(t)$ etc. until $x_i^{(n)}(t) = \frac{\mathrm{d}^n}{\mathrm{d}t^n} x_i(t)$. 
Going forward, we will often drop the time dependence in the notation. 

We consider the following \nth order consensus algorithm:
\begin{equation}
\label{eq:consensuscompact}
u_i = - \sum_{k = 0}^{n-1} a_k\sum_{j \in \mathcal{N}_i} w_{ij}(x_i^{(k)} - x_j^{(k)})
\end{equation}
where the $a_k$ are nonnegative fixed gains.
The feedback in~\eqref{eq:consensuscompact} is termed \emph{relative} as it only based on differences between states of neighboring agents. 
The impact of absolute feedback, where the controllers have access to measurements of the absolute local state, is discussed briefly in Section~\ref{sec:discussion}. 

Defining the full state vector $\xi = [x^{(0)},x^{(1)},\ldots,x^{(n-1)}]^T$, we can write the system's closed-loop dynamics as
\begin{equation}
\label{eq:closedloop}
\ddt \xi = \underbrace{\begin{bmatrix}
0 & I_N & 0 & \cdots & 0\\
0 & 0 & I_N & \cdots &\vdots \\
0 & 0 & 0 & \ddots &\vdots \\
0 & 0 & 0 & \cdots &I_N \\
-a_0L & -a_1L & -a_2L & \cdots & -a_{\mathrm{n-1}}L
\end{bmatrix}}_{\mathcal{A}} \xi,
\end{equation}
where the graph Laplacian~$L$ was defined in~\eqref{eq:laplaciandef} and $I_N$ denotes the $N\times N$ identity matrix.

\begin{rem} \label{rem:dimension}
We limit the analysis to a scalar information state, though an extension to $x_i(t) \in \mathbb{R}^m$ is straightforward provided the system is controllable in the $m$ coordinate directions. In this case, the system dynamics can be written $\dot{\xi} =   (\mathcal{A} \otimes I_m) \xi$, where $\otimes$ denotes the Kronecker product. This would not affect this paper's main result concerning the stability of $\mathcal{A}$.
\end{rem}
\vspace{0.8mm}

\subsubsection{Leader-follower consensus}

We may also consider leader-follower consensus as in~\cite{Swaroop2006}. Here, the state of Agent~1 is assumed fixed, meaning that it acts as a leader for remaining agents (under the assumption that there is a directed path to each of them from agent 1). WLOG we can then set $x_1 = \dot{x}_1 = \ldots,x_1^{n} \equiv 0$. 
The closed-loop dynamics for remaining agents can be written 
\begin{equation}
\label{eq:closedloopred}
\ddt \bar{\xi}= \underbrace{\begin{bmatrix}
0 & I_{N-1} & 0 & \cdots & 0\\
0 & 0 & I_{N-1} & \cdots &\vdots \\
0 & 0 & 0 & \ddots &\vdots \\
0 & 0 & 0 & \cdots & I_{N-1} \\
-a_0\bar{L} & -a_1\bar{L}  & -a_2\bar{L}  & \cdots & -a_{n-1}\bar{L} \\
\end{bmatrix}}_{\bar{\mathcal{A}}} \bar{\xi} ,
\end{equation}
where $\bar{L}$ is the \emph{grounded} graph Laplacian obtained by deleting the first row and column of $L$ and~$\bar{\xi}$ is obtained by removing the states of agent~1. Note that~$\bar{L}$ unlike~$L$ has all of its eigenvalues in the right half plane~\cite{Xia2016}. 

\subsection{Conditions for consensus and admissibility}
\label{sec:reachconsensus}
Consensus among the agents is said to be reached if $x_i^{(k)} \rightarrow x_j^{(k)}$ for all $i,j \in \mathcal{V}$ and for $k = 0,1,\ldots,n-1$. The algorithm~\eqref{eq:consensuscompact} is known to achieve consensus if the eigenvalues of~$\mathcal{A}$ are in the left half plane, apart from exactly $n$ zero eigenvalues associated with the drift of the average states. This condition is in line with standard results for first- and second-order consensus, and is shown in~\cite{Ren2007} for $n = 3$:
\begin{theorem}[\cite{Ren2007}, Theorem 3.1 ]
\label{thm:consensus}
In the case of $n = 3$, the algorithm~\eqref{eq:consensuscompact} achieves consensus exponentially if and only if $\mathcal{A}$ has exactly three zero eigenvalues and all of the other eigenvalues have negative real parts. 
\end{theorem}
We also require the following lemma:
\begin{lemma}[\cite{Ren2007}, Lemma 3.1]
In the case of $n = 3$, the matrix $\mathcal{A}$ has exactly three zero eigenvalues if and only if $L$ has a simple zero eigenvalue.
\end{lemma}
The proofs in~\cite{Ren2007} can be straightforwardly extended to $n>3$.

This means that it is sufficient to verify that the \linebreak $(N-1)\cdot n$ non-zero eigenvalues of~$\mathcal{A}$ have negative real parts. In this paper, we will treat systems where this can be true for small networks, but where at least one eigenvalue leaves the left half plane and causes instability as the network grows beyond some network size~$\bar{N}$. In these cases, we call the control algorithm \emph{inadmissible}. 

\begin{definition}[Admissibility]
A control design $u$ is \emph{admissible} only if the resulting closed-loop system reaches consensus for \emph{any} finite network size $N$. 
\end{definition}
%


\section{Inadmissibility of high-order consensus}
\label{sec:mainresult}
This section is devoted to our main result. We first describe the key underlying assumptions, before proving that the high-order consensus algorithm will be inadmissible if the network graph has what we term a \emph{decreasing algebraic connectivity}. This property applies to several classes of graphs, and we end this section by listing a few of them. 
\subsection{Underlying assumptions}
The following assumptions are important for our analysis. 
\begin{ass}[Locality]
\label{ass:q}
The feedback is \emph{localized}, meaning that the controller uses measurements only from a neighborhood of size at most $q$, where $q$ is fixed and independent of $N$. That is, $ |\mathcal{N}_i| \le q$ for all $i \in \mathcal{V}$.
\end{ass}
\begin{ass}[Finite weights and gains]
\label{ass:finite}
The system gains and edge weights are finite, that is, $w_{ij}\le w_{\max}<\infty$ for all $(i,j) \in \mathcal{E}$ and $a_k\le a_{\max} <\infty $ for all $k = 0,1,\ldots, n$.
\end{ass}
\begin{ass}[Fixed parameters]
\label{ass:fixed}
The gains $a_k$ for all $k = 0,1,\ldots, n$, the maximum edge weight $w_{\max}$, 
and the locality parameter~$q$ do not change if a node (with connecting edges) is added to the graph~$\mathcal{G}$. That is, these parameters are all independent of the total network size~$N$.
\end{ass}
\vspace{1mm}
In the following, the notion of an increase in the network size~$N$ should be understood as the addition of nodes to the network (along with connecting edges) in such a manner that Assumptions~\ref{ass:q}--\ref{ass:fixed} remain satisfied. 
These assumptions contribute to the key property; that the algebraic connectivity of $\mathcal{G}$ decreases towards zero. This is clarified through examples in~Section~\ref{sec:graphs}.

\subsection{Main result}
This paper's main result is negative and states that the consensus algorithm with $n\ge 3$ can never be admissible in certain graphs. 

\begin{theorem}
\label{thm:main}
If $n\ge 3$, no control on the form~\eqref{eq:consensuscompact} is admissible under Assumptions~\ref{ass:q}--\ref{ass:fixed} if the graph $\mathcal{G}$ is such that $\mathrm{Re}\{\lambda_2\} \rightarrow 0$ as $N \rightarrow \infty$. 
\end{theorem}
\begin{proof}
The first step of the proof is to block-diagonalize the matrix~$\mathcal{A}$. Let $U$ be the unitary matrix that diagonalizes the graph Laplacian~$L$, so that $U^*LU = \Lambda = \mathrm{diag}\{0, \lambda_2,\ldots,\lambda_N\}$. By pre- and post-multiplying $\mathcal{A}$ by the $(Nn \times Nn)$ matrix $\mathcal{U} = \mathrm{diag} \{U,U,\ldots,U\}$, we get 
\begin{equation}
\label{eq:diagonalizing}
\mathcal{U}^*\mathcal{A}\mathcal{U} = \underbrace{\begin{bmatrix}
0 & I_{N} & 0 & \cdots & 0\\
0 & 0 & I_{N} & \cdots &\vdots \\
0 & 0 & 0 & \ddots &\vdots \\
0 & 0 & 0 & \cdots &I_{N} \\
-a_0\Lambda & -a_1 \Lambda & -a_2\Lambda & \cdots & -a_{\mathrm{n-1}}\Lambda
\end{bmatrix}}_{\hat{\mathcal{A}}}.
\end{equation}
This can be re-arranged into $N$ decoupled sub-matrices~$\hat{\mathcal{A}}_l$:
\[ \hat{\mathcal{A}}_l = {\begin{bmatrix}
0 & 1 & 0 & \cdots & 0\\
0 & 0 & 1 & \cdots &\vdots \\
0 & 0 & 0 & \ddots &\vdots \\
0 & 0 & 0 & \cdots &1 \\
-a_0\lambda_l & -a_1 \lambda_1 & -a_2\lambda_l & \cdots & -a_{\mathrm{n-1}}\lambda_l
\end{bmatrix}}\]
%
 for $l = 1,\ldots,N$. 
 The eigenvalues of~$\mathcal{A}$ are the union of the eigenvalues of all $\hat{\mathcal{A}}_l$. Clearly, the~$n$ zero eigenvalues are obtained from $\hat{\mathcal{A}}_1$ since $\lambda_1 = 0$. Therefore, we must require all eigenvalues of all $\hat{\mathcal{A}}_l$, $l = 2,\ldots,N$ to have negative real parts for any~$N$ to ensure admissibility. 

The characteristic polynomial of each~$\hat{\mathcal{A}}_l$ is
\begin{equation}
\label{eq:charpoly}
p_l(s) = s^n + a_{n-1} \lambda_l s^{n-1} + \ldots + {a}_1\lambda_ls + a_0\lambda_l.
\end{equation}
In general, the eigenvalues $\lambda_l$ are complex-valued. Consider therefore the Routh-Hurwitz criteria for polynomials with complex coefficients. As these criteria do not appear frequently in literature, we include a detailed derivation here.

Consider the polynomial 
\begin{equation}
\label{eq:examplepoly}
p(\mu) = \mu^n + (f_{n-1} + \mathbf{j} g_{n-1})\mu^{n-1} +\ldots (f_0 + \mathbf{j} g_0) = 0,
\end{equation} 
where $\mathbf{j} = \sqrt{-1}$ denotes the imaginary number.
The roots $\mu$ will be such that $\mathrm{Im}(\mu)>0$ if and only if all inequalities 
 \[-\Delta_2 = -\begin{vmatrix}
1 & f_{n\!-\!1}\\ 0 & g_{n\!-\!1}
\end{vmatrix} >0,~~\Delta_4 = \begin{vmatrix}
1 & f_{n\!-\!1} & f_{n\!-\!2} & f_{n\!-\!3} \\ 0 & g_{n\!-\!1}& g_{n\!-\!2} & g_{n\!-\!3} \\ 0 & 1 & f_{n\!-\!1} & f_{n\!-\!2}\\ 0& 0 & g_{n\!-\!1}& g_{n\!-\!2} \\
\end{vmatrix} >0,  \]
\begin{multline}
\label{eq:Tondlcrit}
\cdots~, ~(-1)^n\Delta_{2n} = \\(-1)^n\begin{vmatrix}
1 & f_{n\!-\!1}& \cdots & f_0 & 0 &\cdots &\cdots & 0\\
0 & g_{n\!-\!1} & \cdots & g_0 & 0& \cdots &\cdots & 0\\
0 & 1 & \cdots & f_{1} & f_0 & 0 & \cdots & 0\\
0 & 0 & \cdots & g_{1} & g_0 & 0 & \cdots & 0\\
&&& \vdots &&& &\\
0 & \cdots & \cdots & 0 & 1 & \cdots & f_{1} & f_0\\
0 & \cdots & \cdots & 0 & 0 & \cdots & g_{1} & g_0\\
\end{vmatrix} >0
\end{multline}
%
%
are satisfied~\cite[pp 21f]{Tondl1965}.  Evaluating the determinants, the first two conditions become \vspace{-0.7mm}
\begin{align} \label{eq:firstcond}
g_{n\!-\!1}&<0,\\
\label{eq:secondcond}
 f_{n\!-\!1}g_{n\!-\!1}g_{n\!-\!2} - f_{n\!-\!2}^{\phantom{2}}g_{n\!-\!1}^2 + g_{n\!-\!3}g_{n\!-\!1} - g_{n\!-\!2}^2 &>0.
\end{align}
We are interested in the polynomial $p_l(s)$ in~\eqref{eq:charpoly} and seek a condition for $\mathrm{Re}\{s\}<0$. 
Such a condition is obtained by substituting $\mu = -\mathbf{j}s$ in~\eqref{eq:examplepoly}, and then identifying the coefficients from~\eqref{eq:charpoly}. 
The coefficients that appear in~\eqref{eq:firstcond}--\eqref{eq:secondcond} are $f_{n-1} = a_{n-1}\mathrm{Im}\{ \lambda_l \},$ $g_{n-1} = -a_{n-1}\mathrm{Re}\{\lambda_l\},$ $f_{n-2} = -a_{n-2}\mathrm{Re}\{\lambda_l \},$ $g_{n-2} = -a_{n-2}\mathrm{Im}\{\lambda_l\},$ $f_{n-3} = -a_{n-3}\mathrm{Im}\{\lambda_l\},$  $g_{n-3} = a_{n-3}\mathrm{Re}\{\lambda_l\}$. Note that these relations hold regardless of~$n$, since the coefficient in front of the the highest order term is 1 in both~\eqref{eq:charpoly} and~\eqref{eq:examplepoly}.

Now, the condition \eqref{eq:firstcond} reads $a_{n-1}\mathrm{Re}\{\lambda_l\}>0$, which is always true for $l = 2,\ldots,N$ if $a_{n-1}>0$, since $\mathrm{Re}\{\lambda_l\}>0$. The condition~\eqref{eq:secondcond} can after some manipulation be written \vspace{-1mm}
\begin{multline}
\label{eq:RHcriterion}
a_{n-1}(\rell)^2(a_{n-1}a_{n-2}\rell - a_{n-3}) + \\ + a_{n-2}(\imll)^2(a_{n-1}^2\rell - a_{n-2}) >0,
\end{multline} 
for $l = 2,\ldots,N$.
While the factors in front of the brackets remain positive for all~$\lambda_l$ (provided $a_k>0$), the brackets will eventually become negative if $\rell \rightarrow 0$ for some $l$. Thus, if $\rel$, where $\lambda_2$ is the  eigenvalue with smallest real part, is decreasing in~$N$ towards zero, the condition~\eqref{eq:RHcriterion} will eventually (i.e., for $N>\bar{N}$ for some finite $\bar{N}$) be violated. 

Therefore, if $n\ge 3$, at least one root of the characteristic polynomial~$p_2(s)$ will have a non-negative real part if $N>\bar{N}$. Theorem~\ref{thm:consensus} is then violated and the control is not admissible. 
\end{proof}
\begin{rem}
If the graph is undirected, then the polynomial~$\eqref{eq:charpoly}$ has real-valued coefficients. The result can then be derived using the standard Routh-Hurwitz criteria. This gives the simpler condition $a_{n-1}a_{n-2}\lambda_2 - a_{n-3} >0$, which can never remain satisfied if $\lambda_2 \rightarrow 0$ as $N\rightarrow \infty$. 
\end{rem}
\begin{rem}
The condition that $L$ be normal can be relaxed if $L$ is diagonalizable as in~\eqref{eq:diagonalizing} by some (non-unitary) matrix. The remainder of the proof would hold true. 
\end{rem}
\vspace{0.8mm}
Theorem~\ref{thm:main} implies that the high-order consensus algorithm can never allow the network size of certain graphs to increase indefinitely without leading to instability. Instability will occur at the smallest~$N$ for which the Routh-Hurwitz criteria~\eqref{eq:Tondlcrit} in the proof are not satisfied, and at least one eigenvalue leaves the open left half plane. We will term this critical network size~$\bar{N}$.  
In Fig.~\ref{fig:criticalN} we display~$\bar{N}$ for $n = 3,4,5$ in a path graph.

\vspace{0.8mm}
\subsubsection{Inadmissibility of high-order leader-follower consensus}
High-order leader-follower consensus on the form~\eqref{eq:closedloopred} in undirected networks will always be inadmissible under the given assumptions. This was also observed in~\cite{Swaroop2006}. 

We first require the following Lemma:
\begin{lemma}
\label{lem:leaderfollowerlemma}
Consider the grounded Laplacian matrix~$\bar{L}$ of an undirected graph~$\mathcal{G}$. 
Let Assumption~\ref{ass:q} hold.
The smallest eigenvalue~$\bar{\lambda}_1$ of~$\bar{L}$ then satisfies
\begin{equation}
\label{eq:lambda1cond}
\bar{\lambda}_1 \le \frac{q}{N-1}w_{\max},
\end{equation}
where $w_{\max}$ is the largest edge weight in $\mathcal{E}$.
\end{lemma}
\begin{proof}
By the Rayleigh-Ritz theorem~\cite[Theorem 4.2.2]{HornJohnson} it holds
\(
\bar{\lambda}_1 \le  \frac{v^T\bar{L}v}{v^Tv} ,~~ \forall v \in \mathbb{C}^{N-1} \backslash \{0\}. \)
This implies in particular that 
\[ \bar{\lambda}_1 \le  \frac{\mathbf{1}_{N-1}^T\bar{L}\mathbf{1}_{N-1}}{\mathbf{1}_{N-1}^T\mathbf{1}_{N-1}} = \frac{\sum_{k \in \mathcal{N}_1} w_{1k}}{N-1} \le \frac{q w_{\max}}{N-1},\]
where $\mathbf{1}_{N-1}^T\bar{L}\mathbf{1}_{N-1} = \sum_{k \in \mathcal{N}_1} w_{1k}$ is the weight sum of all edges connected to the leader node 1.  This is true since each row~$k$ of~$\bar{L}$ sums to zero if node~$k$ has no connection to the node 1, and otherwise to~$w_{1k}\le w_{\max}$.
\end{proof}

\vspace{0.8mm}
\begin{theorem}
\label{thm:leaderfollower}
Assume that the graph~$\mathcal{G}$ is undirected. The leader-follower consensus algorithm represented in~\eqref{eq:closedloopred} is inadmissible for $n \ge 3$ under Assumptions~\ref{ass:q}--\ref{ass:fixed}.
\end{theorem}
\begin{proof}
The arguments in the  proof of Theorem~\ref{thm:main} apply. In this case,  
$N-1$ real-valued characteristic polynomials like~\eqref{eq:charpoly} are obtained. 
The condition~\eqref{eq:RHcriterion} reduces to
\vspace{-1mm}
\begin{equation}
\label{eq:reducedcond}
a_{n-1}a_{n-2}\bar{\lambda}_l - a_{n-3} >0, \vspace{-0.8mm}
\end{equation}
for $l = 1,\ldots,N-1$. 
Using Lemma~\ref{lem:leaderfollowerlemma}, we see that~\eqref{eq:reducedcond} requires
$a_{n-1}a_{n-2} > \frac{1}{q w_{\max}}a_{n-3}(N-1)$, which cannot stay satisfied for large~$N$. The algorithm is thus inadmissible. 
%
%
%
%
\end{proof}

\subsection{Affected classes of graphs}
\label{sec:graphs}
The inadmissibility of high-order consensus applies to any network whose underlying graph is such that $\rel$ is decreasing towards zero as~$N$ increases. The second-smallest Laplacian eigenvalue $\lambda_2$ of an undirected graph is real-valued and known as the \emph{algebraic connectivity} of the graph. While the correct generalization of algebraic connectivity to directed graphs is not clear-cut, we know the following:
\begin{lemma}
\label{lem:mirrorlemma}
If~$L$ is normal then
\[\rel = \lambda_2^s,\]
where $\lambda_2^s$ is the smallest non-zero eigenvalue of $L^s = (L + L^T)/2$, that is, the symmetric part of~$L$. 
\end{lemma}
\begin{proof}
See Lemma~9.1.2 in \cite{GuoThesis}.
\end{proof}
The matrix~$L^s$ is the graph Laplacian corresponding to the \emph{mirror graph} $\hat{\mathcal{G}}$ of $\mathcal{G}$, which is the \emph{undirected} graph obtained as $\hat{\mathcal{G}} = \{\mathcal{V}, \mathcal{E} \cup \hat{\mathcal{E}}\}$, where $\hat{\mathcal{E}}$ is the set of all edges in $\mathcal{E}$, but \emph{reversed}, and whose edge weights are $\hat{w}_{ij} = \hat{w}_{ij} = (w_{ij}+w_{ji})/2$ \cite{OlfatiSaber2004}. Clearly, the mirror graph of an undirected graph is the graph itself. 
Lemma~\ref{lem:mirrorlemma} implies that $\rel$ of $\mathcal{G}$ is the algebraic connectivity of its mirror graph $\mathcal{\hat{G}}$. 

We introduce the following terminology:
\begin{definition}
The graph~$\mathcal{G}$ is said to have \emph{decreasing algebraic connectivity} if for its mirror graph~$\hat{\mathcal{G}}$, the algebraic connectivity $\lambda_2 \rightarrow 0$ as $N \rightarrow \infty$. 
\end{definition}

This means that Theorem~\ref{thm:main} will apply to graphs with decreasing algebraic connectivity, and it suffices to identify this property in undirected graphs. We next give a (non-exhaustive) account of classes of graphs with this property.

\vspace{0.5mm}
\subsubsection{Lattices, fuzzes, and their embedded graphs}
Consider the $d$-dimensional toric lattice~$\mathbb{Z}_M^d$ with $N = M^d$ nodes, and let each node be connected to its $r$ neighbors in each direction (letting $q = 2rd$). Such a lattice is called an $r$-fuzz.
\begin{lemma}[Algebraic connectivity of $r$-fuzz]
\label{lem:fuzz}
In the $r$-fuzz lattice of $d$ dimensions
\begin{equation}
\label{eq:lambda2fuzz}
\lambda_2 \le \frac{c}{N^{2/d}},
\end{equation}
where $c$ is a constant that depends on the fixed parameters~$r$ (that is, $q$), $w_{\max}$ and~$d$, but not on $N$.
\end{lemma}
\begin{proof}
Follows the derivations in~\cite{Tegling2017b}. 
\end{proof}
The bound~\eqref{eq:lambda2fuzz} also holds for any subgraph of the $r$-fuzz lattice, that is, any graph that is \emph{embeddable} in the lattice. This follows from the following important lemma:
\begin{lemma}
Adding an edge to a graph or increasing the weight of an edge increases (or leaves unchanged)~$\lambda_2$ of the corresponding graph Laplacian, and vice versa.
\end{lemma}
\begin{proof}
See~\cite[Theorem 3.2]{Mohar1991} for addition of an edge. 

Increasing the weight of an edge $(i',j')$ by $\Delta w$ means that the new graph Laplacian can be written $L' = L + \Delta L$, where $\Delta L$ is also a positive semidefinite graph Laplacian (of a disconnected graph). By~\cite[Theorem 2.8.1]{BrouwerBook} this implies that $\lambda_l' \ge \lambda_l$ for each $l = 1,\ldots,N$, and in particular~$\lambda_2' \ge \lambda_2$.
\end{proof}

\vspace{0.5mm}
\subsubsection{Planar graphs}
Planar graphs are embeddable in two-dimensional lattices so Lemma~\ref{lem:fuzz} applies. For this important case, however, a more precise bound is available:
\begin{lemma}[Algebraic connectivity of planar graphs]
For a planar graph,
\vspace{-1mm}
\begin{equation}
\lambda_2 \le \frac{8qw_{\max}}{N},
\end{equation}
\end{lemma}
\vspace{1mm}
\begin{proof}
See \cite[Theorem 6]{Spielman2007}.
\end{proof}

\vspace{0.5mm}
\subsubsection{Constant-genus graphs}
Planar graphs can be generalized to graphs with constant \emph{genus}. The genus of a planar graph is $g =0$. Higher genus implies that the graph can be drawn on a surface with $g$ handles (or ``holes'') without any one edge crossing another. For example, a torus would correspond to $g =1$ and a pretzel shape to $g = 3$. 
\begin{lemma}[Alg. connectivity of constant-genus graphs]
Let $\mathcal{G}$ have constant and bounded genus $g$. Then 
\vspace{-1mm}
\[\lambda_2  \le  \frac{c_2}{N} , \]\vspace{-1mm}
where $c_2$ is a constant that depends on $q$, $g$ and $w_{\max}$, but not on~$N$.
\end{lemma}
\begin{proof}
See \cite[Theorem 2.3]{Kelner2006}.
\end{proof}

\vspace{0.5mm}
\subsubsection{Tree graphs with growing diameter}
The \emph{diameter} $\mathrm{diam}\{\mathcal{G}\}$ of the graph~$\mathcal{G}$ is defined as the longest distance between any pair of its nodes. For tree graphs where the diameter is growing, we can state the following lemma:
%
\begin{lemma}[Algebraic connectivity of tree graphs]
Let $\mathcal{G}$ be a tree graph. Then
\vspace{-1mm}
\begin{equation}
\label{eq:tree}
\lambda_2 \le \frac{\pi^2 w_{\max}}{(\mathrm{diam}\{\mathcal{G}\} +1 )^2},
\end{equation}\vspace{-1mm}
and if $\mathrm{diam}\{\mathcal{G}\} \rightarrow \infty $ as  $N \rightarrow \infty$, then $\lambda_2 \rightarrow 0$.
\end{lemma}\vspace{0.5mm}
\begin{proof}
Follows from \cite[Corollary 4.4]{Grone1990}, noting that $1-\cos x \le \frac{x^2}{2}$ for any~$x$. Clearly, the right hand side is decreasing in $\mathrm{diam}\{\mathcal{G}\}$.
\end{proof}

\begin{figure}
\centering
\includegraphics[scale=1]{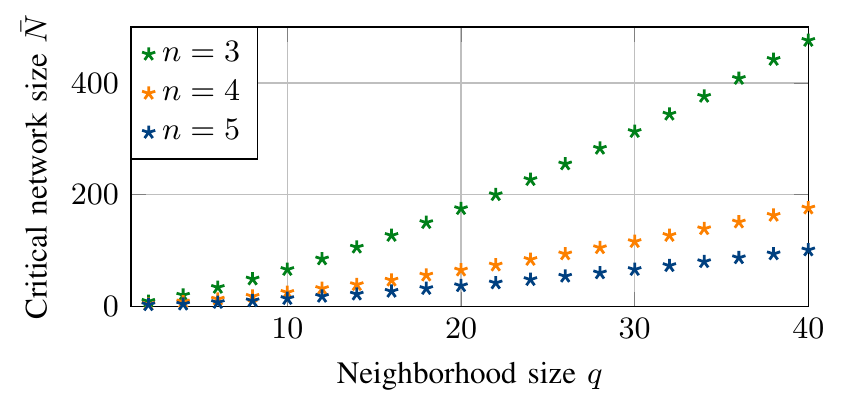}
\caption{Critical network size~$\bar{N}$ at which the system is destabilized with an \nth order consensus algorithm. The graph is an undirected path where each node is connected to its $q$ nearest neighbors. Increasing~$q$ here increases~$\bar{N}$ faster than linearly. Also note that for higher $n$, the system is destabilized at smaller~$\bar{N}$. 
 }
\label{fig:criticalN}
\end{figure}

\begin{figure*}[h!]
  \centering
  \subfloat[][{Network graph}]{
  \includegraphics[width = 0.24\textwidth]{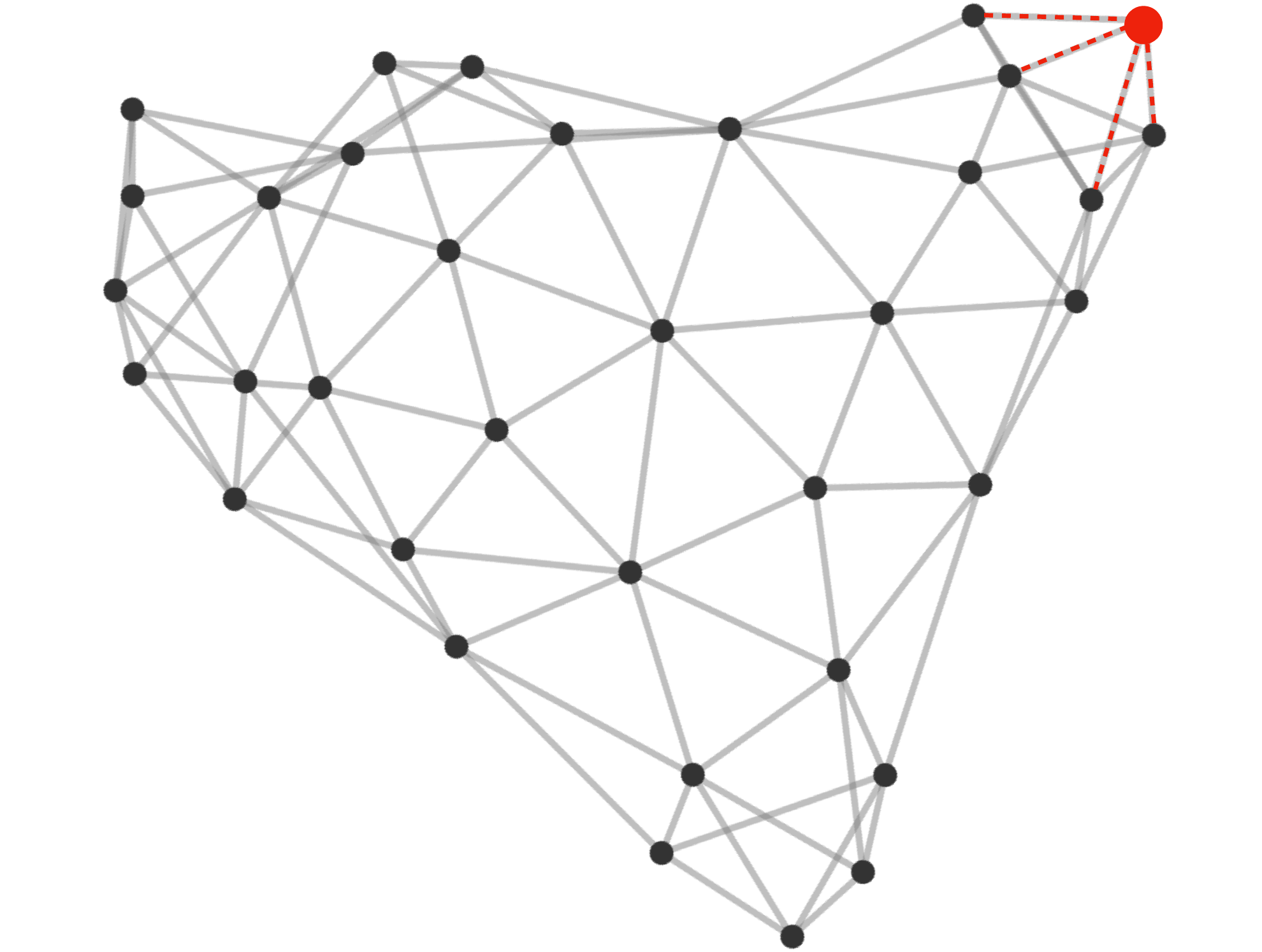}
  \label{fig:graph}  }
  \subfloat[][{$N=34$}] {
  \includegraphics[width = 0.37\textwidth]{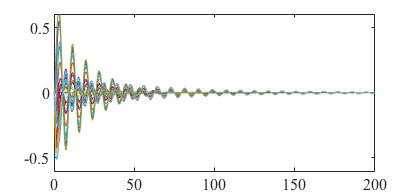}
  \label{fig:34node}}
   \subfloat[][{ $N=35$}] {
  \includegraphics[width = 0.37\textwidth]{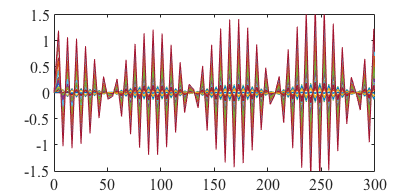}
  \label{fig:35node}} \\
       \caption{
{Simulation of $3^{\mathrm{rd}}$ order consensus in graph depicted in (a), subject to random initial accelerations. In (b) the network's 34 agents converge to an equilibrium. In (c) a $35^{\mathrm{th}}$ node has been added, indicated by red color in the graph. This addition leads to instability. The plots (b) and (c) show position trajectories relative to Agent no. 1. Note the different scales.   } }
\label{fig:simulation}
\end{figure*}

\section{Numerical examples}
\label{sec:numerical}
We next present two simple numerical examples to illustrate the paper's main result.

\subsection{Locality and critical network size}
Assumption~\ref{ass:q} of locality, that is, a fixed upper bound on the size of each agent's neighborhood, is key for our main result. Indeed, if each agent's neighborhood were allowed to grow as more and more agents are added to the network, the high-order consensus algorithm could stay admissible. 

As an example, consider an undirected path graph where each node is connected to its $q/2$ neighbors in each direction. If all edge weights are 1, its algebraic connectivity is
\[ \lambda_2 =  \sum_{\substack{k = 1  } }^{q/2} 2\left(1- \cos \left(\frac{\pi k}{N} \right) \right),\]
which for any given~$N$ is larger, the greater $q$ is. Increasing~$q$ thus delays the violation of the stability criteria as $N$ grows. 

In Fig.~\ref{fig:criticalN}, we depict the critical network size~$\bar{N}$ as a function of the neighborhood size~$q$ in the undirected path graph. Here, we have selected a consensus algorithm where $a_0 = 0.1$, $a_1 = 0.8$, $a_2,a_3,a_4 = 1$. The plot shows that increasing~$q$ pushes $\bar{N}$ faster than linearly. 

We also note that the system becomes unstable at smaller~$\bar{N}$ for higher~$n$. This is because the higher-order conditions in~\eqref{eq:Tondlcrit} are violated sooner than the lower-order ones (though we only required the violation of one condition, $\Delta_4 > 0$, to prove inadmissibility in Theorem~\ref{thm:main}). 

\subsection{Instability through node addition}
The second example illustrates the phase transition -- from consensus to instability -- that the system experiences as the critical network size is reached. Fig.~\ref{fig:graph} illustrates a graph that has been randomly generated by means of triangulation. Here, the maximum neighborhood size is $q = 8$ while the median is~5. All edge weights are set to 1. 

We consider a third order consensus algorithm: 
\[x_i^{(3)} = -\sum_{j \in \mathcal{N}_i}\left[  0.5({x}_i - {x}_j)  + (\dot{x}_i - \dot{x}_j)+ (\ddot{x}_i - \ddot{x}_j) \right], \]
which by~\eqref{eq:RHcriterion} ensures stability if $\lambda_2>0.5$. With 34 nodes, the graph in Fig.~\ref{fig:graph} has $\lambda_2^{34} = 0.536$ and the system achieves consensus, as seen from the simulation in Fig.~\ref{fig:34node}. We then add a $35^{\mathrm{th}}$ node along with 4 connecting edges, as indicated in red color in the graph in Fig.~\ref{fig:graph}. Now, $\lambda_2^{35} = 0.493$ and the system becomes unstable.
Fig.~\ref{fig:35node} shows how the agents' positions~$x$ oscillate at an increasing amplitude.



\section{Conclusions}
\label{sec:discussion}
This paper's results show that there is an important difference between the well-studied standard first- and second-order consensus algorithms, and the corresponding higher-order algorithm, in that the latter is not always scalable to large networks. In classes of graphs where the algebraic connectivity decreases towards zero due to a locality constraint, it will cause instability at some finite network size.
%
%
%
An interesting consequence of this result is that the addition of only one agent to a given multi-agent network can render a previously converging system unstable. 

A further interesting consequence is that an analysis of asymptotic (in network size) performance of localized, consensus-like feedback is only possible in first- and second-order integrator networks. This means that the analysis in~\cite{Bamieh2012} cannot, as was conjectured there, be extended to chains of $n>2$ integrators. 

The model assumed in this work -- the integrator chain -- is standard in the consensus literature. A valid question is, however, to which extent our result applies to more general dynamics. Such a discussion requires a distinction between relative feedback, as considered here, and absolute feedback from local states (such as damping). It is part of ongoing work to characterize the local dynamical properties under which consensus remains (in)admissible.
Preliminary results indicate that absolute feedback of high-derivative states (e.g. acceleration if $n = 3$) would be necessary for admissibility. 

%
%


\section*{Acknowledgement}
We wish to thank Richard Pates and Swaroop Darbha for their insightful comments and a number of interesting discussions. 

\bibliographystyle{IEEETran}
\bibliography{emmasbib2015,BassamBib,consrefs}
\vspace{1mm}

\end{document}